\documentclass{article}[12pt]
\usepackage{geometry}
\usepackage{graphicx,amsthm,natbib,mathtools,amsfonts}
 \usepackage{mathrsfs}
\usepackage{multirow,mfirstuc}
\usepackage{subfigure}
\usepackage{algorithm}
\usepackage{algorithmic}
\usepackage{cleveref}
\usepackage{xcolor}
\newtheorem{theorem}{Theorem}

\newtheorem{assumption}[theorem]{Assumption}
\newtheorem{proposition}[theorem]{Proposition}
\newtheorem{definition}[theorem]{Definition}
\newtheorem{lemma}[theorem]{Lemma}

\newcommand{\Acal}{{\cal A}}

\newcommand{\Ical}{{\cal I}}

\newcommand{\Lcal}{{\cal L}}

\newcommand{\Scal}{{\cal S}}

\newcommand{\sign}{\text{sign}}   
  
\begin{document} 
\title{\LARGE Convergence Rate Analysis of Proximal  Iteratively Reweighted \MakeLowercase{$\ell_1$} Methods for  \MakeLowercase{$\ell_p$} Regularization Problems}

\author{
{\large
Hao Wang\thanks{Email: haw309@gmail.com}
} \\[2pt]
{\large School of Information Science and Technology, ShanghaiTech
University}
\\[6pt]
{\large
Hao Zeng\thanks{Email: zenghao@shanghaitech.edu.cn}
} \\[2pt]
{\large School of Information Science and Technology, ShanghaiTech
University}
\\[6pt]
{\large  and}\\[6pt]
{\large  Jiashan Wang}\thanks{Email: jsw1119@gmail.com}\\[2pt]
\large Department of Mathematics, University of Washington
}

\date{\large\today}

\maketitle

\begin{abstract}
 	{ In this paper,  we focus on  the local convergence rate analysis of the proximal iteratively reweighted $\ell_1$ algorithms for solving $\ell_p$ regularization
		problems, which are widely applied for inducing sparse solutions. 
		We show that if the  \emph{Kurdyka-\L ojasiewicz} (KL)  property  is satisfied,  the algorithm converges 
		to  a unique first-order stationary  point; furthermore, the algorithm  has local  linear convergence   or local sublinear convergence. The theoretical results we derived are much stronger than the existing results for iteratively reweighted $\ell_1$ algorithms.}
	
\end{abstract}

\section{Introduction}

In recent years, sparse optimization problems arises in a wide range of fields including machine learning, image processing and compressed sensing \citep{portilla2009image,figueiredo2007majorization, candes2008enhancing, chun2010sparse, yin2015minimization, scardapane2017group}.  A common  technique to enforce sparsity is to add the $\ell_p$ ($0< p < 1)$ regularization term to the objective function, which is called the $\ell_p$ regularized problem 
\begin{equation}\label{prob.lpproblem}\tag{P}
\begin{aligned}
& \underset{x\in \mathbb{R}^n}{\min} 
& & F(x):= f(x)+ \lambda \|x\|_p^p   \quad \text{with }  \|x\|_p^p:= \sum_{i=1}^n |x_i|^p
\end{aligned}
\end{equation}
where $f:\mathbb{R}^n\to \mathbb{R}$ is a continuously differentiable function, $p\in(0,1)$ 
and $\lambda > 0$ is the regularization parameter.
It is generally believed that $\ell_p$ can have  superior ability to induce sparse solutions of a system   compared with traditional convex regularization  techniques. For example, when $p\to 0$, this problem approximates the $\ell_0$-norm optimization problem, that is usually useful for image processing; when $p= 1$, that is the well-known $\ell_1$-norm regularized problem. 

However, it is full of challenges to seek the solution of  $\ell_p$-norm optimization problems due to the nonconvex and nonsmooth propery of $\ell_p$-norm.  In fact, \citep{ge2011note} proved that finding the global minimal value of the problem with $\ell_p$-norm   regularization term is strongly NP-Hard.

Recently, effective methods have been proposed to construct smooth approximation models for 
the $\ell_p$ regularization problem.  Some works \citep{candes2008enhancing,lu2014iterative,Chen13} focus on constructing Lipshcitz continuous approximation to replace $|x_i|^p$. Other works \citep{chen2010convergence} 
and \citep{lai2011unconstrained}  take the smoothing technique which adds  perturbation to each $|x_i|$ to form 
the $\epsilon$-approximation of the $\ell_p$-norm. In the later case, the approximate objective function becomes
\begin{equation}\label{l2-lp-appro1}
f(x)+ \lambda \sum_{i=1}^n \left( |x_i| + \epsilon \right)^p, 
\end{equation}
with $\epsilon>0$. 
Iteratively reweighted $\ell_1$ methods \citep{lu2014iterative,sun2017global,wang2018nonconvex} were proposed for solving 
approximation \eqref{l2-lp-appro1}.
At each iteration, it replaces each component of  the $\epsilon$-approximation  via linearizing $(\cdot)^p$ at $x^{k}$, i.e., 
\begin{equation}\label{first.l2}
p(|x_i^k|+\epsilon_i)^{p-1}|x_i|.
\end{equation}
There is a tradeoff in the choice of $\epsilon$. Large $\epsilon$  smoothes out many local minimizers, while small values make the subproblems difficult to solve due to bad local minimizers. In order to approximate \eqref{prob.lpproblem} effectively, \citep{lu2014iterative} improved these weights by dynamically updating perturbation parameter $\epsilon_i$ at each iteration.  
Recently, it is shown in \cite{wang2019relating} that the general framework of iteratively reweighted $\ell_1$ methods is equivalent to solving a weighted $\ell_1$ regularization problem, based on which the global convergence and $O(1/k)$ worst-case complexity of optimality residual were analyzed.

In this paper, we focus on the local convergence rate analysis of the proximal iteratively reweighted $\ell_1$ methods for the $\ell_p$ regularization problem.  This type of algorithms was first presented and  investigated in \cite{lu2014proximal} with fixed $\epsilon > 0$ and there was no  convergence rate established.   Our purpose is to show that local \emph{linear convergence} or \emph{sublinear convergence} can   be obtained under mild assumptions.  The Kurdyka-\L ojasiewicz   (K\L) property  \citep{boltedanlew17, Bolte2014}  is generally believed to  capture a
broad spectrum of the local geometries that a nonconvex function can have and  has been shown to hold
ubiquitously for most practical functions.  
It has been  exploited extensively to analyze the convergence rate of various first-order algorithms for
nonconvex optimization   \citep{Attouch2009, li2017convergence,Bolte2014, zhou2016convergence}. 
However, it has not been exploited to establish the convergence rate of iteratively reweighted methods. 
In this paper, we exploit the K\L\  property of $f$ to provide
a comprehensive study of the convergence rate of iteratively reweighted  $\ell_1$ methods  for  $\ell_p$ regularization  problems. We anticipate our study to substantially advance the existing understanding of the convergence of iteratively reweighted methods to a much
broader range of  nonconvex regularization problems.

\subsection{Notation}
We denote $\mathbb{R}$ and $\mathbb{Q}$ as the set of real numbers and rational numbers.
In $\mathbb{R}^n$, denote $\|\cdot\|_p$ as the $\ell_p$ norm with $p\in(0,+\infty)$, i.e., 
$\|x\|_p = \left(\sum_{i=1}^n |x_i|^p\right)^{1/p}$.   Note  that for $p\in(0,1)$, 
this  does not define a proper norm due to its lack of subadditivity. 
If function $f: \mathbb{R}^n \to \bar{ \mathbb{R}}:=\mathbb{R} \cup \{+\infty\}$
is convex, then the subdiferential of $f$ at $\bar x$ is given by 
\[ \partial f(\bar x):=\{ z \mid f(\bar x) + \langle z, x-\bar x\rangle \le f(x),  \  \forall x\in\mathbb{R}^n\}.\]
In particular, for $x\in \mathbb{R}^n$, we use $\partial \|x\|_1$ to denote the set 
$\{ \xi \in \mathbb{R}^n \mid  \xi_i \in \partial |x_i|, i=1,\ldots, n\}.$

Given a lower semi-continuous function $f$,  the limiting subdifferential at $a$ is defined as 
\[ \bar\partial f(a) :=\{ z^* = \lim_{x^k\to a,  f(x^k)\to f(a)} z^k, \  z^k\in\partial_F f(x^k)\}.\]
The Frechet subdifferential of $f$  at $a$ defined as 
\[ \partial_F f(a):=\{z\in\mathbb{R}^n \mid \liminf_{x\to a } \frac{f(x)-f(a)-\langle z, x-a\rangle}{\|x-a\|_2} \ge 0\}. \]
The Clarke subdifferential $\partial_c f$ is the convex hull of the limiting subdifferential. 
It holds true that $\partial f(a) \subset \bar\partial f(a) \subset \partial_cf(a)$.  
For convex functions, $\partial f(a) = \partial_F f(a) = \bar \partial f(a) = \partial_c f(a)$ 
and for differentiable $f$, $\partial f(a) = \partial_F f(a) = \bar \partial f(a) = \partial_c f(a) = \{\nabla f(a)\}$.

For $f: \mathbb{R}^n \to \mathbb{R}$ and index sets $\Acal$ and $\Ical$ satisfying
$\Acal \cup \Ical  = \{1, \ldots, n\}$,  
let $f(x_{\Ical})$ be  the function in the reduced space $\mathbb{R}^{|\Ical|}$ 
by fixing $x_i = 0, i\in \Acal$. 
For $a, b\in\mathbb{R}^n$,  $a\le b$ means the inequality holds for each component, i.e., 
$a_i \le b_i $ for $i=1,\ldots, n$. 
For closed convex set $\chi \subset \mathbb{R}^n$, define the Euclidean distance of point $a\in\mathbb{R}^n$
to $\beta$ as $\text{dist}(a,  \chi) = \min_{b\in \chi} \| a - b\|_2$. 
Let $\{-1,0,+1\}^n$ be the set of vectors in $\mathbb{R}^n$ filled with elements in $\{-1,0,+1\}$. 
The support of $x\in\mathbb{R}^n$ is defined as 
$\Ical(x) := \{ i \mid x_i \neq 0 \}$.  For $a, b\in \mathbb{R}$, let  $a\bmod b$ denote the remainder of $a$ divided by $b$. 

\section{Proximal iteratively reweighted $\ell_1$ method}
\label{sec.l1lp}

In this section, we present the Proximal Iteratively Reweighted $\ell_1$ (PIRL1) methods and examine their properties when applied to
\eqref{prob.lpproblem}.  The PIRL1  method is based on the smoothed approximation of $F$ by adding perturbation $\epsilon_i$ to each component 
of $|x|$
\[ F(x, \epsilon): = f(x) + \lambda \sum_{i=1}^n (|x_i| + \epsilon_i)^p,\]
where $\epsilon \in \mathbb{R}^n_{++}$ is the perturbation vector.   
At the $k$th iteration, PIRL1 solves the subproblem
\[  \min_x  \nabla f(x^k)^T (x-x^k)  + \frac{\beta}{2}\|x-x^k\|^2 +\lambda \sum_{i=1}^{n}w_i^k|x_i|\] 
with $\beta > L_f/2$ and the weight $w_i^k$ is defined as
$w_i^k:= p(|x_i^k|+\epsilon_i^k)^{p-1}$
with $\epsilon_i \to 0$.

The framework of the PIRL1  is presented in Algorithm \ref{alg.framework}.

\begin{algorithm}[H]
	\caption{Proximal Iteratively Reweighted $\ell_1$  Methods (PIRL1)}
	\label{alg.framework}
	\begin{algorithmic}[1]
		\STATE{\textbf{Input:} $\mu\in(0,1)$, $\beta > L_f/2$, $\epsilon^0\in\mathbb{R}^n_{++}$ and $x^0 $.  Set $k=0$}
		\REPEAT		
		\STATE Compute weights:  $w_i^k = p(|x^k_i|+\epsilon^k_i)^{p-1}.$ 
		\STATE Compute new iterate:
		\begin{equation}
		\begin{aligned}
		x^{k+1}  & \gets \underset{x\in \mathbb{R}^n}{\text{argmin}} \   \big\{   \nabla f(x^k)^T (x-x^k)   + \frac{\beta}{2}\|x-x^k\|^2 +\lambda \sum_{i=1}^{n}w_i^k|x_i|\big\}.
		\end{aligned}
		\end{equation}
		\STATE Choose $\epsilon^{k+1}\le \mu \epsilon^k$.
		\STATE Set $k\gets k+1$. 
		\UNTIL{convergence}
	\end{algorithmic} 
\end{algorithm}

We make the following assumptions about  the functions in  \eqref{prob.lpproblem}. 
formulation
\begin{assumption}\label{ass.lip}
	$f$ is Lipschitz differentiable with constant $L_f \ge 0$.
	The initial point $(x^0, \epsilon^0)$  is  such that    
	$ \Lcal(F^0):= \{ x \mid F(x)  \le F^0 := F(x^0, \epsilon^0)\} $ is contained in a bounded ball 
	$\mathbb{B}_R:= \{ x \mid \|x\|_2 \le R\}$. 
\end{assumption}

\subsection{Basic properties}

The first-order necessary optimality condition2 of \eqref{prob.lpproblem}  is given 
\cite{lu2014iterative}, 
\begin{equation} \label{eq:optimalcondition}
\nabla_if(x^*)+\lambda p |x_i^*|^{p-1} \sign( x_i^*) =0 \quad \text{for} \quad i\in\Ical(x^*),
\end{equation}
which is equivalent to 
\begin{equation} \label{eq:optimalcondition2}
x_i \nabla_if(x^*)+\lambda p |x_i^*|^p =0 \quad \text{for} \quad i = 1, \ldots, n.
\end{equation}
We call any point satisfying \eqref{eq:optimalcondition} or  \eqref{eq:optimalcondition2}  is stationary for $F(x,0)$.

\begin{proposition}\label{prop.sign}
	Assume $\{x^k\}$ is generated by  Algorithm \ref{alg.framework}  and     Assumption \ref{ass.lip} holds. Let  $\Gamma$ be the cluster point set  of $\{x^k\}$. We have the following
	\begin{enumerate}
		\item[(a)] $F(x^{k+1}, \epsilon^{k+1})\leq F(x^k,\epsilon^k) - \hat\beta \|x^{k+1}-x^k\|_2^2$ with $\hat\beta:=\beta- \frac{L_f}{2}$ and 
		$\{x^k\}\subset \Lcal(F^0) \subset \mathbb{B}_R$.
		\item[(b)]$\exists \text{ constant } \zeta$ such that $F(x^*, 0) = \zeta$, $ \forall x^* \in \Gamma$.  
		\item[(c)]$\sum\limits_{k=0}^{\infty}\|x^{k+1}-x^k\|_2^2< + \infty$.
		\item[(d)]All points in $\Gamma$ are stationary for $F(x,0)$.
	\end{enumerate}
\end{proposition}
\begin{proof}
	(a).
	Lipschitz differentiability of $f$ gives
	\begin{equation}\label{th.lip1}
	f(x^{k+1})  \leq f(x^{k}) + \nabla f(x^k)^T(x^{k+1}-x^k) + \frac{L_f}{2} \|x^{k+1}-x^k \|_2^2
	\end{equation}
	The concavity of $a^p$ on $\mathbb{R}_{++}$ gives
	$a_1^p \le a_2^p + pa_2^{p-1}(a_1- a_2)$ for any $a_1, a_2 \in \mathbb{R}_{++}$. Hence we have
	\[\begin{aligned}
	(|x_i^{k+1}|+\epsilon_i^k)^p \le &\ (|x_i^k |+\epsilon_i^k)^p + p (|x_i^k |+\epsilon_i^k)^{p-1} (|x_i^{k+1} | - |x_i^k|)\\
	= &\ (|x_i^k |+\epsilon_i^k)^p +  w_i^k (|x_i^{k+1} | - |x_i^k|).
	\end{aligned} \] 
	Summing the above inequality over  $i$ yields 
	\begin{equation}\label{w les w} 
	\sum_{i=1}^n ( |x_i^{k+1}|+\epsilon_i^k)^p   \le  \sum_{i=1}^n (|x_i^k |+\epsilon_i^k)^p +  \sum_{i=1}^n w_i^k (|x_i^{k+1} | - |x_i^k|). 
	\end{equation}
	The optimality condition of subproblems implies there exists $\xi^{k+1} \in \partial \|x^{k+1}\|_1$ such that 
	\begin{equation}\label{kkt subproblem}
	\nabla f(x^k) + \beta^k (x^{k+1}-x^k) + \lambda  w^k\circ\xi^{k+1}  = 0.
	\end{equation}
	The definition of subgradient implies
	$|y_i| \leq |x_i| + \xi_i(y_i-x_i) $ with $\xi_i\in\partial |y_i|$. Thus, we have
	\begin{equation}\label{eq:th1}
	\begin{aligned}
	&\ F(x^{k+1},\epsilon^{k+1}) - F(x^k,\epsilon^k) \\
	= &\ f(x^{k+1}) +\lambda \sum_{i=1}^{n}(|x_i^{k+1}|+\epsilon_i^k)^p - \bigg(f(x^k) +\lambda \sum_{i=1}^{n}(|x_i^k|+\epsilon_i^k)^p\bigg) \\
	\leq &\ \nabla f(x^k)^T(x^{k+1}-x^k) + \frac{L_f}{2} \|x^{k+1}-x^k \|_2^2 + \lambda \sum_{i=1}^n w_i^k (|x_i^{k+1} | - |x_i^k|) \\
	\leq &\ \nabla f(x^k)^T(x^{k+1}-x^k) +\frac{L_f}{2} \|x^{k+1}-x^k \|_2^2 + \lambda \sum_{i=1}^n w_i^k\xi_i^{k+1}(x_i^{k+1}-x_i^k) \\
	= &\ (\nabla f(x^k)+\beta (x^{k+1}-x^k) + \lambda  w^k\circ\xi^{k+1} )^T(x^{k+1}-x^k)  -(\beta - \frac{L_f}{2})\|x^{k+1}-x^k\|_2^2 \\
	= &\ -(\beta- \frac{L_f}{2})\|x^{k+1}-x^k\|_2^2,
	\end{aligned}
	\end{equation}
	where the first inequality follows from \eqref{th.lip1} and \eqref{w les w} and the last equality is due to \eqref{kkt subproblem}. 	
	Therefore, (a) holds true with $\hat\beta = \beta- L_f/2$. 
	
	

	(b). Monotonicity of $\{F(x^k, \epsilon^k)\}$ gives $\zeta:= \lim\limits_{k\to \infty \atop k\in\Scal} F(x^k,\epsilon^{k}) = F(x^*, 0)$ for any $x^*\in \Gamma$ with subsequence $\{x^k\}_{\Scal}\to x^*$.

	(c).  From (a), we have 
	\[
	\hat\beta \sum_{k=0}^t \|x^{k+1}-x^k\|_2^2 \leq F(x^{0},\epsilon^{0})-F(x^{t+1}, \epsilon^{t+1}).
	\]
	Then, taking the limit as $t \to \infty$, 
	\[
	\hat\beta \sum_{k=0}^{\infty}\|x^{k+1}-x^k\|_2^2 \leq F(x^{0},\epsilon^{0})-\lim_{t\to \infty} F(x^{t+1},\epsilon^{t+1})<\infty.
	\]
	
	(d). 
	Let $x^*$ be a limit point with $\{x^k\}_{\Scal} \to x^*$. 
	The optimal condition of the $k$th subproblem  implies  
	\[
	\nabla_if(x^{k-1}) + \beta(x_i^{k}-x_i^{k-1})+ \lambda p(|x_i^{k-1}|+\epsilon_i^{k-1})^{p-1}\text{sign}(x_i^{k})  = 0,  \quad\forall i\in \Ical(x^k).
	\]
	Taking the limit on $\Scal$, we have 
	for each $   i\in \Ical(x^*)$, 
	\[\begin{aligned}
	0  = &  \lim_{k\to\infty \atop k\in\Scal} \nabla_if(x^{k-1}) + \beta(x_i^k -x_i^{k-1})+ \lambda p(|x_i^{k-1}|+\epsilon_i^{k-1})^{p-1} \text{sign}(x_i^k)  \\
	= &  \lim_{k\to\infty \atop k\in\Scal}  \nabla_if(x^k) + \beta(x_i^{k+1}-x_i^k)+ \lambda p|x_i^*|^{p-1} \text{sign}(x_i^*)  \\
	= &     \lim_{k\to\infty \atop k\in\Scal}	 \nabla_if(x^k) +\lambda p|x_i^*|^{p-1} \text{sign}(x_i^*) \\
	=   &   \nabla_if(x^*) + \lambda p |x_i^*|^{p-1} \sign(x_i^*). 
	\end{aligned} \]
	Here the second equality is from $\epsilon_i^k\to 0$ for all $i\in \Ical(x^*)$.
	Therefore, $x^*$ is a stationary point of $F(x,0)$.

\end{proof}

%
%
%
%
%
%
%
%

Algorithm \ref{alg.framework} belongs to the framework of iteratively reweighted $\ell_1$  methods proposed 
in \cite{wang2019relating}.   
From \cite{wang2019relating}, the following properties  hold true.

\begin{theorem}\cite[Theorem 1]{wang2019relating}\label{thm.stable.support} Assume   Assumption \ref{ass.lip}   holds and let $\{(x^k, \epsilon^k)\}$ be a sequence generated by   Algorithm \ref{alg.framework}.  Define constant
	$C=\sup_{x\in\mathbb{B}_R}\|\nabla f(x)\|_2 + 2R\beta$.  Then we have the following
	\begin{enumerate}
		\item[(i)]   If $w(x_i^{\tilde k}, \epsilon_i^{\tilde k}) > C/\lambda$ for some ${\tilde k}\in \mathbb{N}$,  then  $x_i^k \equiv 0$ for all $k > \tilde k$. Conversely, 
		if there exists $\hat k>\tilde k$ for any $\tilde k \in\mathbb{N}$ such that $x_i^{\hat k} \neq 0$,  then $w_i^k \le  C/\lambda$ for all $k\in\mathbb{N}$. 
		\item[(ii)]       There  exist index sets $\Ical^*\cup \Acal^* = \{1,\ldots,n\}$ and $\bar k > 0$, such that  $\forall \  k> \bar k$,  
		$\Ical(x^k)\equiv \Ical^*$ and $\Acal(x^k) \equiv \Acal^*$.
		\item[(iii)] 	For any $i\in\Ical^*$,  there holds that 
		\begin{equation}\label{eq.xboundeps} |x_i^k| > \left(\frac{C}{p \lambda }\right)^{\frac{1}{p-1}} - \epsilon_i^k > 0,\quad   i\in\Ical^*. \end{equation} 
		Therefore, $\{|x_i^k|, i\in\Ical^*, k\in\mathbb{N}\}$ are bounded away from 0 after some $\hat k \in \mathbb{N}$. 
		\item[(iv)]  For any cluster point $x^*$ of $\{x^k\}$, it holds that $\Ical(x^*) = \Ical^*$, $\Acal(x^*) = \Acal^*$ and  
		\begin{equation}\label{eq.xbound} |x^*_i| \ge   \left(\frac{C}{p \lambda }\right)^{\frac{1}{p-1}},\quad i\in \Ical^*. \end{equation}
	\end{enumerate}
\end{theorem}

The above theorem shows locally the support of the iterates remains unchanged and the nonzeros are bounded away from 0.  The 
next theorem shows that the signs of iterates stay stable locally.

\begin{theorem}\cite[Theorem 2]{wang2019relating}\label{thm.stable.sign}  
	Let $\{x^k\}$ be a sequence generated by   Algorithm \ref{alg.framework} and    Assumption \ref{ass.lip} is satisfied.  
	There exists $\bar k  \in \mathbb{N}$, such that 
	the sign of $\{x^k\}$ are fixed for all $k > \bar k$, i.e.,  
	$\text{sign}(x^k)\equiv s$ for some  $s \in \{-1,0,+1\}^n$. 
\end{theorem}

\subsection{Kurdyka-\L ojasiewicz property}\label{sec.KL}

\cite{attouch2013convergence} have proved a series of convergence results of descent methods for semi-algebraic  problems under the assumption that the objective satisfies the Kurdyka-\L ojasiewicz (KL) property. In fact, this assumption covers a wide range of problems such as nonsmooth semi-algebraic minimization problem \citep{Bolte2014}.  The definition of  KL property  is given below.

\begin{definition}[Kurdyka-\L ojasiewicz property]\label{df:KL}
	The function $f:\mathbb{R}^n\to \mathbb{R}\cup\{+\infty\}$ is said to have the Kurdyka-\L ojasiewicz property at $x^*\in \text{dom}\bar\partial f$ if there exists $\eta\in (0,+\infty]$, a neighborhood $U$ of $x^*$ and a continuous concave function $\phi:[0,\eta) \to \mathbb{R}_+$ such that:
	\begin{enumerate}
		\item[(i)] $\phi(0)=0$,
		\item[(ii)] $\phi$ is $C^1$ on $(0,\eta)$,
		\item[(iii)] for all $s\in(0,\eta)$, $\phi'(s)>0$,
		\item[(iv)] for all $x$ in $U\cap [f(x^*)<f<f(x^*)+\eta]$, the Kurdyka-\L ojasiewicz inequality holds
		\[
		\phi'(f(x)-f(x^*))\text{dist}(0, \bar\partial f(x))\geq 1.
		\]
	\end{enumerate}   
\end{definition}
If $f$ is smooth, then condition (iv) reverts to \cite{attouch2013convergence}
\[\| \nabla (\phi\circ f)(x)\| \ge 1.\]

Since for sufficiently large $k$, the iterates $\{x^k_{\Ical^*}\}$ remains in the same orthant of $\mathbb{R}^{|\Ical^*|}$ and are bounded away from
the axis, or equivalently,  \[ \{x^k_{\Ical^*} \}\in \Omega \subset  \mathbb{R}^{|\Ical^*| }_s \] where 
$\Omega$ is in the interior of an orthant and is bounded away from the axis.  
To further analyze the property of iterates $\{(x^k,\epsilon^k)\}$,  denote $\delta_i = \sqrt{\epsilon_i}$ .
Therefore, we can write $F(x, \delta)$ as a function of $(x, \delta)$  for simplicity. 
We can assume the reduced function $F(x_{\Ical^*}, \delta_{\Acal^*})$ has the KL property at $(x^*_{\Ical^*}, 0_{\Ical^*})$.  
In fact, we only need to make assumption on $f$.  To see this, we introduce the concept of semi-algebraic functions, which 
is a weak condition and can cover most common functions.

\begin{definition}[Semi-algebraic functions] A subset of $\mathbb{R}^n$ is called semi-algebraic if it can be written as 
	a finite union of sets of the form 
	\[\{ x\in \mathbb{R}^n : h_i(x) = 0, \  q_i(x)<0, \  i=1,\ldots, p\},\]
	where $h_i, q_i$ are real polynomial functions.  A function $f: \mathbb{R}^n \to \mathbb{R}\cup\{+\infty\}$ is semi-algebraic 
	if its graph is a semi-algebraic subset of $\mathbb{R}^{n+1}$.  
\end{definition}

Semi-algebraic functions satisfy KL property with $\phi(x) = cs^{1-\theta}$, for some $\theta\in[0,1)\cap \mathbb{Q}$ and some $c> 0$  \citep{boltedanlew17,bolte2007clarke}. This non-smooth result 
generalizes the famous {\L}ojasiewicz inequality for real-analytic function \citep{lojasiewicz1963propriete}.  Finite sums of semi-algebraic 
functions are semi-algebraic; for $p\in\mathbb{Q}$, 
$\sum_{i\in\Ical^*} (|x_i|+\epsilon_i)^p$ is semi-algebraic  around $(x^*_{\Ical^*}, 0_{\Acal^*})$ by \citep{wakabayashi2008remarks}.   
Therefore, we only need to assume 
$f(x_{\Ical^*})$ is semi-algebraic in a neighborhood around $x^*$.

We state this assumption formally below. 

\begin{assumption}\label{ass.KL}  Suppose $p\in\mathbb{Q}$ and $f(x_{\Ical^*})$ is semi-algebraic in $ \mathbb{R}^{|\Ical^*| }_s$, where $x^*$  is a limit point of $\{x^k\}$ generated by the PIRL1 methods. 
\end{assumption} 



For simplicity of the following analysis and without loss of generality,  we assume $\Ical^* = \{1, ..., n\}$ and $\Acal^* = \emptyset$, so that 
for sufficiently large $k$, 
the iterates $\{x^k_{\Ical^*}\}$ remains in the same orthant  are bounded away from
the axis.

\section{The uniqueness of limit points}

We investigate the uniqueness of limit points 
under KL property of $F$.

\begin{lemma}\label{lem.acc.gc2}
	Let $\{x^k\}$ be a sequence generated by Algorithm \ref{alg.framework}. 
	The following statements hold.
	\begin{enumerate}
		\item[(i)]  There exists $D_1>0$ such that for all $k$
		\begin{equation*}\label{eq:acc.gc4}
		\|\nabla F(x^k, \delta^k)\|_2 \leq D_1(\|x^k-x^{k-1}\|_2 +\|\delta^{k-1}\|_1-\|\delta^{k}\|_1),
		\end{equation*}
		and $\lim\limits_{k\to \infty} \|\nabla F(x^k,\delta^k)\|_2=0$. 
		\item[(ii)]   $\{F (x^k, \delta^k)\}$ is monotonically decreasing, and there exists   $\hat\beta>0$ such that 
		\begin{equation*}\label{eq:acc.gc5}
		F(x^{k+1},\delta^{k+1})-F(x^k, \delta^k) \geq \hat\beta\|x^{k+1}-x^{k}\|^2_2.
		\end{equation*}
		\item[(iii)]  $F( x^*,0)= \zeta = \lim\limits_{k\to\infty}F(x^k,\delta^{k})$ for all $x^*\in\Gamma$,  where $\Gamma$ is the set of 
		the cluster points of  $\{x^k\}$. 
	\end{enumerate}
\end{lemma}

\begin{proof}
	(i)
	The gradient of  $F$ at $(x^k, \delta^k)$ is 
	\begin{equation}
	\label{grad.xydelta}
	\begin{aligned}
	\nabla_x F(x^k,\delta^k) & =\nabla f(x^k)+\lambda w^k\circ  \text{sign}(x^{k}), \\
	\nabla_\delta F(x^k, \delta^k) & = 2\lambda w^k \circ \delta^k.
	\end{aligned}
	\end{equation}
	
	We first derive an upper bound for $\|\nabla_x F(x^k,\delta^k)\|_2$. The first-order optimality condition of the $(k-1)$th subproblem at $x^k$ is 
	
	\begin{equation*}
	\label{eq:acc.gc2}
	\nabla f(x^{k-1})+\beta^k(x^{k}-x^{k-1})+\lambda w^{k-1}\circ \text{sign}(x^{k}) = 0. 
	\end{equation*}
	Hence, we have
	\begin{equation}
	\label{eq:acc.gc3}
	\begin{aligned}
	\nabla_xF(x^k,\delta^k) 
	=  \nabla f(x^k)-\nabla f(x^{k-1}) - \beta^k(x^{k}-x^{k-1})+\lambda (w^k-w^{k-1})\circ \text{sign}(x^{k}).
	\end{aligned}
	\end{equation}
	By the Lipschitz property of $f$, the first two terms in  \eqref{eq:acc.gc3} is bounded by 
	\[ \| \nabla f(x^k)-\nabla f(x^{k-1}) - \beta^k(x^{k}-x^{k-1}) \|_2
	\le (L_f + \beta  )\|x^k-x^{k-1}\|_2. \]
	Now we give an upper bound for the third term.
	It follows from Lagrange's mean value theorem that $\exists \ z_i^k$ between $|x_i^k|+(\delta_i^{k})^2$ and $|x_i^{k-1}|+(\delta_i^{k-1})^2$, such that
	\begin{equation*}
	\begin{aligned}
	\big| (w_i^k-w_i^{k-1})\cdot \text{sign}(x_i^k) \big|&= \big| w_i^k-w_i^{k-1}\big| \\
	& =   \big|p(|x_i^k|+(\delta_i^{k})^2)^{p-1}-p(|x_i^{k-1}|+(\delta_i^{k-1})^2)^{p-1}\big|\\
	& =   \big|p(1-p)(z_i^k)^{p-2} (|x_i^k|-|x_i^{k-1}|+(\delta_i^{k})^2-(\delta_i^{k-1})^2)\big|\\
	&\leq   p(1-p)(z_i^k)^{p-2} (|x_i^k-x_i^{k-1}|+ (\delta_i^{k-1})^2-(\delta_i^{k})^2 )\\
	&\leq   p(1-p)(z_i^k)^{p-2} (|x_i^k-x_i^{k-1}|+ 2\delta_i^0(\delta_i^{k-1}-\delta_i^{k} )) \\
	&\leq   p(1-p) \left( \frac{p\lambda}{C}\right)^{\frac{p-2}{1-p}} (|x_i^k-x_i^{k-1}|+ 2\delta_i^0(\delta_i^{k-1}-\delta_i^{k} )),
	\end{aligned}
	\end{equation*}	
	where the first equality is by the fact that $x_i^k\neq 0$ and the last inequality by observing the following. From Theorem \ref{thm.stable.support}(i), we know  
	\begin{equation}
	\begin{aligned}	
	|x_i^k|+(\delta_i^{k})^2 = (\frac{w_i^k}{p})^{\frac{1}{p-1}} \ge (\frac{C}{p\lambda})^{\frac{1}{p-1}} =  (\frac{p\lambda}{C})^{\frac{1}{1-p}} \\
	|x_i^{k-1}|+(\delta_i^{k-1})^2 = (\frac{w_i^{k-1}}{p})^{\frac{1}{p-1}} \ge (\frac{C}{p\lambda})^{\frac{1}{p-1}} = (\frac{p\lambda}{C})^{\frac{1}{1-p}} ,
	\end{aligned}
	\end{equation}
	hence 
	\[
	(z_i^k)^{p-2} \le  \left( \frac{p\lambda}{C}\right)^{\frac{p-2}{1-p}}.
	\]
	Now we can obtain an upper bound for the third term in  \eqref{eq:acc.gc3},
	\begin{equation}\label{eq.nablaf.3}
	\begin{aligned}
	\| (w^k-w^{k-1})\circ \text{sign}(x^k) \|_2 
	\leq  &\ 	\| (w^k-w^{k-1})\circ \text{sign}(x^k) \|_1 \\
	= & \sum_{i=1}^{n} \big| (w_i^k-w_i^{k-1})\cdot \text{sign}(x_i^k) \big| \\
	\leq & \sum_{i=1}^n  p(1-p) \left( \frac{p\lambda}{C}\right)^{\frac{p-2}{1-p}} (|x_i^k-x_i^{k-1}|+ 2\delta_i^0(\delta_i^{k-1}-\delta_i^{k} )) \\
	\leq   &\ \bar D \left(\|x^k-x^{k-1}\|_1 + 2\|\delta^0\|_\infty( \|\delta^{k-1}\|_1 - \|\delta^{k}\|_1)\right)\\
	\leq   &\  \bar D\left(\sqrt{n} \|x^k-x^{k-1}\|_2 + 2\|\delta^0\|_\infty( \|\delta^{k-1}\|_1 - \|\delta^{k}\|_1)\right),	\end{aligned}
	\end{equation}
	where  $\bar D : = p(1-p)\left( \tfrac{p\lambda}{C}\right)^{\frac{p-2}{1-p}}$.  Putting together the bounds for all three terms in  \eqref{eq:acc.gc3}, we have 
	\begin{equation}
	\label{first.bound}
	\begin{aligned}
	\| \nabla_xF(x^k,\delta^k)\|_2 
	\le &\  (L_f +  \beta  )\|x^k-x^{k-1}\|_2  +  2\bar D\|\delta^0\|_\infty (\|\delta^{k-1}\|_1 - \|\delta^{k}\|_1).
	\end{aligned}
	\end{equation}
	On the other hand, \begin{equation}\label{third.bound}
	\begin{aligned}
	\| \nabla_\delta F(x^k, \delta^k) \|_2 & \le \| \nabla_\delta F(x^k, \delta^k) \|_1\\
	& = \sum_{i=1}^n 2\lambda w_i^k \delta_i^k\\
	& \le \sum_{i=1}^n 2\lambda \tfrac{C}{\lambda} \frac{\sqrt{\mu}}{1-\sqrt{\mu}} (\delta_i^{k-1}-\delta_i^k)\\
	& \leq  \frac{2 C\sqrt{\mu}}{1-\sqrt{\mu}}(\|\delta^{k-1}\|_1-\|\delta^k\|_1), 
	\end{aligned}	
	\end{equation}	
	where the second inequality is by Theorem \ref{thm.stable.support}(i) and $\delta^k \leq \sqrt{\mu} \delta^{k-1}$. 
	Overall,  we obtain from \eqref{first.bound} and \eqref{third.bound} 
	that Part (i) holds true by setting 
	\[ D_1 = \max \left( \beta  + L_f,
	2\bar C \|\delta^0\|_\infty+  \tfrac{2 C\sqrt{\mu}}{1-\sqrt{\mu}} \right).
	\]
	
	Part (ii) and (iii) follows directly from Proposition \ref{prop.sign}(a) 
	and  \ref{prop.sign}(b), respectively. 
\end{proof}

Now we are ready to prove the global convergence under KL property.

\begin{theorem}\label{lem.acc.gc2.conv}
	Let $\{x^k\}$ be a sequence generated by Algorithm \ref{alg.framework} and $F$ is a KL function at   $(x^*,0)$ with $x^*\in\Gamma$.
	Then    $\{x^k\}$ converges to a stationary point of $F(x,0)$; moreover, $$\sum_{k=0}^{\infty}\|x^{k+1}-x^k \|_2<\infty.$$
\end{theorem}

\begin{proof}
	By Proposition \ref{prop.sign}, every cluster point is stationary for $F(x,0)$,  it is sufficient to show that 
	$\{x^k\}$ has a unique cluster point.  

	By Lemma \ref{lem.acc.gc2},   $F(x^k,\delta^{k})$ is monotonically decreasing and converging to $\zeta$.  
	If $F(x^k,\delta^{k})= \zeta$ after some $k_0$, then from Lemma \ref{lem.acc.gc2}(ii), 
	we know $x^{k+1} = x^k$ for all $k > k_0$, meaning $x^k \equiv x^{k_0} \in \Gamma$, so that the proof is done.

	We next consider the case that $ F(x^k,\delta^{k})>\zeta$ for all $k$. Since $F$ has the KL property at every 
	$(x^*,0)\in \bar\Gamma$, there exists a continuous concave function $\phi$ with $\eta >0$ and 
	neighborhood $U=\{(x,\delta)\in \mathbb{R}^n \times \mathbb{R}^n: \text{dist}((x,\delta),\bar\Gamma) <\tau \}$      such that 
	\begin{equation}\label{eq:eiskl}
	\phi'(F(x,\delta)-\zeta)\text{dist}((0,0),\nabla F(x,\delta))\geq 1
	\end{equation} 
	for all $(x,\delta)\in U \cap\{(x,\delta)\in  \mathbb{R}^n \times \mathbb{R}^n: \zeta< F(x,\delta) < \zeta +\eta \}$.

	Let  $\bar \Gamma \subset\mathbb{R}^{2n}$ be the set of limit points of $\{(x^k, \delta^{k})\}$, i.e., $\bar\Gamma:=\{(x^*,0)\mid x^*\in\Gamma\}$, by Proposition \ref{prop.sign}(ii), we have
	\[
	\lim_{k\to \infty} \text{dist}((x^k,\delta^{k}),\bar\Gamma) =0.
	\]
	Hence, there exist $k_1\in\mathbb{N}$ such that $\text{dist}((x^k,\delta^{k}),\bar \Gamma)<\tau$ for any $k>k_1$. On the other 
	hand, since  $\{F(x^k,\delta^{k})\}$ is monotonically decreasing  and converges to $\zeta$, there exists $k_2\in\mathbb{N}$ such that $\zeta < F(x^{k},\delta^{k}) <\zeta +\eta$ for all $k>k_2$. Letting $\bar{k}=\max \{k_1,k_2 \}$ and noticing that 
	$F$ is smooth at $(x^k,  \delta^k)$ for all $k>\bar k$,  we know from \eqref{eq:eiskl} that 
	\begin{equation}\label{eq:klp}
	\phi'(F(x^k,\delta^{k})-\zeta) \|\nabla F(x^k,\delta^{k})\|_2 \geq 1, \qquad \text{for all } k\geq \bar{k}. 
	\end{equation}
	It follows that for any $k\geq \bar{k}$, 
	\[
	\begin{aligned}
	& \Big[\phi( F(x^{k},\delta^{k})-\zeta)-\phi(  F(x^{k+1},\delta^{k+1})-\zeta)\Big] \cdot D_1(\|x^k-x^{k-1}\|_2+ \|\delta^{k-1}\|_1 - \|\delta^k\|_1)\\
	\geq\ & \Big[\phi( F(x^{k},\delta^{k})-\zeta)-\phi(  F(x^{k+1},\delta^{k+1})-\zeta)\Big] \cdot \|\nabla F(x^k,\delta^{k})\|_2\\
	\geq\ &   \phi'(F(x^k,\delta^{k})-\zeta)\cdot \|\nabla F(x^k,\delta^{k})\|_2\cdot \Big[F(x^{k},\delta^{k})-F(x^{k+1},\delta^{k+1})\Big]\\
	\geq\ & F(x^k,\delta^{k})-F(x^{k+1},\delta^{k+1})\\
	\geq\ & \hat\beta \|x^{k+1}-x^{k}\|_2^2,
	\end{aligned}
	\]
	where the first inequality is by Lemma \ref{lem.acc.gc2}(i), the second inequality is by the concavity of $\phi$, 
	and the third inequality is by \eqref{eq:klp} and the last inequality is by Lemma \ref{lem.acc.gc2}(ii). 
	Rearranging and taking the square root of both sides,  and using the inequality of arithmetic 
	and geometric means inequality, we have
	\begin{equation*}
	\begin{aligned}
	\|x^{k}-x^{k+1}\|_2 \le &\ \sqrt{\frac{2D_1}{\hat\beta}[\phi(F(x^k,\delta^{k})-\zeta)-\phi( F(x^{k+1},\delta^{k+1})-\zeta)]}\\ &\times \sqrt{\frac{\|x^k-x^{k-1}\|_2+
			(\|\delta^{k-1}\|_1 - \|\delta^{k}\|_1)}{2}}\\
	\leq &\ \frac{D_1}{\hat\beta} \Big[\phi(F(x^k,\delta^{k})-\zeta)-\phi( F(x^{k+1},\delta^{k+1})-\zeta)\Big]\\ &
	\ + \frac{1}{4}\Big[\|x^k-x^{k-1}\|_2
	+ (\|\delta^{k-1}\|_1-\|\delta^{k}\|_1)\Big]. 
	\end{aligned}
	\end{equation*}
	Subtracting $\frac{1}{4}\|x^k-x^{k+1}\|_2$ from both sides, we have 
	\begin{equation*} 
	\begin{aligned}
	\frac{3}{4}\|x^{k+1}-x^{k}\|_2  \leq &\frac{D_1}{\hat\beta}\Big[\phi(F(x^k,\delta^{k})-\zeta)-\phi( F(x^{k+1},\delta^{k+1})-\zeta)\Big] \\
	&+\frac{1}{4}( \|x^k-x^{k-1}\|_2- \|x^{k+1}-x^{k}\|_2+ \|\delta^{k-1}\|_1-\|\delta^{k}\|_1).
	\end{aligned}
	\end{equation*}
	Summing up both sides from $\bar k$ to $t$, we have 
	\begin{equation*}\label{eq:acc.gc6}
	\begin{aligned}
	\frac{3}{4}\sum_{k=\bar k}^t \|x^{k+1}-x^{k}\|_2  \leq &\frac{D_1}{\hat\beta}\Big[\phi(F(x^{\bar k},\delta^{\bar k})-\zeta)-\phi( F(x^{t+1},\delta^{t+1})-\zeta)\Big] \\
	&+\frac{1}{4}(\|x^{\bar{k}}-x^{\bar{k}-1}\|_2 -\|x^{t+1}-x^{t}\|_2 +\|\delta^{\bar k-1}\|_1-\|\delta^{t}\|_1).
	\end{aligned}
	\end{equation*}
	Now letting $t\to\infty$, we know $\|\delta^{t}\|_1\to 0$ and $\|x^{t+1}-x^{t}\|_2\to 0$ by  Proposition \ref{prop.sign}(c), 
	and that $\phi( F(x^{t+1},\delta^{t+1})-\zeta)\to \phi(\zeta-\zeta) = \phi(0) = 0$.
	Therefore, we have 
	\begin{equation}\label{sum.x.bound}
	\sum_{k=\bar k}^\infty \|x^{k+1}-x^{k}\|_2 \le \frac{4D_1}{3\hat\beta}\phi(F(x^{\bar k},\delta^{\bar k})-\zeta) 
	+  \frac{1}{3}(\|x^{\bar{k}}-x^{\bar{k}-1}\|_2 + \|\delta^{\bar k-1}\|_1) < \infty.
	\end{equation}
	Hence $\{x^k\}$ is a Cauchy sequence, and consequently 
	it is a convergent sequence. 
\end{proof}

%

\section{Local convergence rate}
We have shown that there is only one unique limit point of $\{x^k\}$ under KL property.  Now we investigate the local convergence rate of  Algorithm \ref{alg.framework} by assuming that $\phi$ in the KL definition taking the form $\phi(s)=cs^{1-\theta}$ for some 
$\theta\in[0,1)$ and $c> 0$. By the discussion in \S\ref{sec.KL}, this additional  requirement is satisfied by the semialgebraic functions, 
which is also commonly satisfied by a wide range of functions.   
\begin{theorem}
	\label{th:convergerate}
	Suppose $\{x^k\}$ is   generated by Algorithm \ref{alg.framework}  and converges to $x^*$. Assume that $F$ is a KL function with $\phi$ in the KL definition taking the form $\phi(s)=cs^{1-\theta}$ for some $\theta \in [0,1)$ and $c> 0$.  Then the following statements hold.
	\begin{enumerate}
		\item[(i)] If $\theta =0$, then there exists $k_0\in\mathbb{N}$ so that $x^k\equiv x^*$   for any $k>k_0$;
		\item[(ii)] If $\theta \in (0,\frac{1}{2}]$, then there exist $\gamma \in (0,1), c_1>0$ such that 
		\begin{equation}\label{linear.1}
		\|x^k-x^*\|_2< c_1\gamma^k
		\end{equation} for sufficiently large $k$;
		\item[(iii)] If $\theta \in (\frac{1}{2},1)$, then there exist $c_2>0$ such that 
		\begin{equation}\label{linear.2}
		\|x^k-x^*\|_2< c_2 k^{-\frac{1-\theta}{2\theta-1}}
		\end{equation}   for sufficiently large $k$.
	\end{enumerate}
\end{theorem}
\begin{proof} 
	(i)  If $\theta=0$, then $\phi(s)=cs$ and $\phi'(s)\equiv c$.  We claim that there must exist $k_0>0$ such that $F(x^{k_0},\delta^{k_0})=\zeta$. Suppose by contradiction this is 
	not true so that $F(z^k)>\zeta$ for all $k$. Since $\lim\limits_{k\to \infty } x^k=x^*$ and the sequence $\{F(x^k,\delta^{k})\}$ is monotonically decreasing to $\zeta$ by Lemma\ref{lem.acc.gc2}.  
	The KL inequality implies that  all sufficiently large $k$, 
	\[
	c \|\nabla F(x^k,\delta^{k})\|_2 \geq 1,
	\]
	contradicting  $\|\nabla F(x^k,\delta^{k})\|_2 \to 0$ by Lemma \ref{lem.acc.gc2}(i).  
	Thus, there exists $k_0\in\mathbb{N}$ such that $F(x^{k},\delta^k)= F(x^{k_0},\delta^{k_0})=\zeta$ for all $k>k_0$. Hence, we conclude from  Lemma\ref{lem.acc.gc2}(ii) that $x^{k+1} = x^k$ for all $k > k_0$, meaning $x^k \equiv x^* = x^{k_0}$   for all $k \ge k_0$. This proves (i).

	(ii)-(iii) Now consider $\theta\in (0,1)$. First of all, if there exists $k_0\in\mathbb{N}$ such that $F(x^{k_0}, \delta^{k_0})=\zeta$, then using the same argument of the proof for (ii),  we can see that $\{x^k\}$ converges finitely.   
	Thus, we only need to consider the case that $F(x^k, \delta^k)>\zeta$ for all $k$.
	
	Define $S^k=\sum_{l=k}^{\infty}\|x^{l+1}-x^l\|_2$. It holds that 
	\[ \|x^k - x^*\|_2 = \|x^k - \lim_{t\to\infty} x^t\|_2 =   \|\lim_{t\to\infty}\sum_{l=k}^t( x^{l+1} - x^l)\|_2 \le \sum_{l=k}^{\infty}\|x^{l+1}-x^l\|_2=S^k.\]
	Therefore, we only have to prove $S^k$ also has the same upper bound as in \eqref{linear.1} and  \eqref{linear.2}. 
	
	To derive the upper bound for $S^k$,  	
	by KL inequality with $\phi'(s)=c(1-\theta)s^{-\theta}$, for $k > \bar k$,
	\begin{equation}
	\label{eq:rate2}
	c(1-\theta)(F(x^k,\delta^k)-\zeta)^{-\theta} \| \nabla F(x^k,\delta^{k}))\|_2 \geq 1.
	\end{equation}
	On the other hand, using \eqref{eq:acc.gc4}(i) and the definition of $S^k$, we see that for all sufficiently large $k$,
	\begin{equation}
	\label{eq:rate3}
	\| \nabla F(x^k,\delta^{k}))\|_2 \leq D_1(S^{k-1}-S^k+\|\delta^{k-1}\|_1-\|\delta^{k}\|_1)
	\end{equation}
	Combining \eqref{eq:rate2} with \eqref{eq:rate3}, we have
	\[
	(F(x^k,\delta^k)-\zeta)^{\theta}\leq D_1  c(1-\theta) (S^{k-1}-S^k+\|\delta^{k-1}\|_1-\|\delta^{k}\|_1).
	\]
	Taking a power of $(1-\theta)/\theta$ to both sides of the above inequality and scaling both sides by $c$, we obtain that for all $k> \bar k$
	\begin{equation}\label{eq:rate.5.27}
	\begin{aligned}
	\phi(F(x^k,\delta^k)-\zeta)&=c\Big[F(x^k,\delta^k)-\zeta\Big]^{1-\theta}\\
	&\leq c  \Big[D_1    c(1-\theta) (S^{k-1}-S^k+\|\delta^{k-1}\|_1-\|\delta^{k}\|_1)\Big]^{\frac{1-\theta}{\theta}}\\
	&\leq c  \Big[D_1    c(1-\theta) (S^{k-1}-S^k+\|\delta^{k-1}\|_1)\Big]^{\frac{1-\theta}{\theta}},
	\end{aligned}
	\end{equation} 
	From   \eqref{sum.x.bound}, we have 
	\begin{equation}\label{eq:rate1}
	\begin{aligned}
	S^k 
	\le \frac{4D_1}{3\hat\beta}\phi(F(x^k, \delta^k)-\zeta) + \frac{1}{3}( \|x^k - x^{k-1}\|_2 +  \|\delta^{k-1}\|_1).
	\end{aligned}
	\end{equation}
	Combining \eqref{eq:rate.5.27} and    \eqref{eq:rate1}, we have 
	\begin{equation}\label{eq:rate5}
	\begin{aligned}
	S^k &\leq C_1[S^{k-1}-S^k+\|\delta^{k-1}\|_1]^{\frac{1-\theta}{\theta}} + \frac{1}{3}(S^{k-1}-S^k +  \|\delta^{k-1}\|_1) \\
	&\leq C_1[S^{k-2}-S^k+\|\delta^{k-1}\|_1]^{\frac{1-\theta}{\theta}} + \frac{1}{3}[S^{k-2}-S^k+\|\delta^{k-1}\|_1]
	\end{aligned}
	\end{equation}
	where $C_1=\frac{4D_1c}{3\hat\beta} \left(D_1\cdot c(1-\theta)\right)^{\frac{1-\theta}{\theta}}$.
	It follows that  
	\begin{equation}\label{eq.47rate}
	\begin{aligned}
	&\ S^k + \frac{\sqrt{\mu}}{1-\mu}\|\delta^k\|_1 \\
	\leq &\ C_1[S^{k-2}-S^k+\|\delta^{k-1}\|_1]^{\frac{1-\theta}{\theta}} + \frac{1}{3}[S^{k-2}-S^k+\|\delta^{k-1}\|_1]+ \frac{\sqrt{\mu}}{1-\mu}\|\delta^k\|_1\\
	\leq&\  C_1[S^{k-2}-S^k+\|\delta^{k-1}\|_1]^{\frac{1-\theta}{\theta}} + \frac{1}{3}[S^{k-2}-S^k+\|\delta^{k-1}\|_1]+ \frac{\mu}{1-\mu}\|\delta^{k-1}\|_1\\
	\leq&\ C_1 [S^{k-2}-S^k+\|\delta^{k-1}\|_1]^{\frac{1-\theta}{\theta}}+ C_2[S^{k-2}-S^k+\|\delta^{k-1}\|_1],
	\end{aligned}
	\end{equation}
	with  
	$
	C_2:=\frac{1}{3}+\frac{\mu}{1-\mu}
	$
	and the second inequality is by the update $\delta^k \le \sqrt{\mu} \delta^{k-1}$.

	For part (ii), $\theta \in (0, \frac{1}{2}]$.   Notice that 
	\[\frac{1-\theta}{\theta}\geq 1\ \text{ and }\  S^{k-2}-S^k+\|\delta^{k-1}\|_1 \to 0.\]
	Hence, there exists sufficient large $k$ such that 
	\[ \Big[S^{k-2}-S^k+\|\delta^{k-1}\|_1\Big]^{\frac{1-\theta}{\theta}}  \le  S^{k-2}-S^k+\|\delta^{k-1}\|_1,\]
	we assume the above inequality holds for all $k \ge \bar k$.
	This,  combined with \eqref{eq.47rate}, yields 
	\begin{equation} \label{eq:skrel}
	S^k + \frac{\sqrt{\mu}}{1-\mu}\|\delta^k\|_1 \le (C_1+C_2) \Big[S^{k-2}-S^k+\|\delta^{k-1}\|_1\Big]
	\end{equation}
	for any $k\ge \bar k$. Using $\delta_k \le \mu \delta_{k-1}$, we can show that
	\begin{equation}\label{eq:delta}
	\delta^{k-1}\leq \frac{\sqrt{\mu}}{1-\mu}(\delta^{k-2}-\delta^k).
	\end{equation}
	Combining \eqref{eq:skrel} and \eqref{eq:delta} gives
	\[\begin{aligned}
	S^k + \frac{\sqrt{\mu}}{1-\mu}\|\delta^k\|_1  &\le (C_1+C_2)\bigg[\big(S^{k-2}+\frac{\sqrt{\mu}}{1-\mu}\|\delta^{k-2}\|_1\big)-\big(S^{k}+\frac{\sqrt{\mu}}{1-\mu}\|\delta^{k}\|_1
	\big)\bigg].
	\end{aligned}
	\]
	Rearranging this inequality gives
	\[\begin{aligned}
	S^k + \frac{\sqrt{\mu}}{1-\mu}\|\delta^k\|_1  \le & \frac{C_1+C_2}{C_1+C_2+1}\bigg[S^{k-2}+\frac{\sqrt{\mu}}{1-\mu}\|\delta^{k-2}\|_1\bigg]\\
	\le & \left( \frac{C_1+C_2}{C_1+C_2+1} \right)^{ \lfloor \frac{k}{2} \rfloor } \bigg[S^{k\bmod 2}+\frac{\sqrt{\mu}}{1-\mu}\|\delta^{k\bmod 2}\|_1\bigg]\\
	\le & \left( \frac{C_1+C_2}{C_1+C_2+1} \right)^{ \frac{k-1}{2} } \bigg[S^0+\frac{\sqrt{\mu}}{1-\mu}\|\delta^0\|_1\bigg].
	\end{aligned}
	\]
	Therefore, for any $k\ge \bar k$,
	\[ \|x^k-x^*\|_2 \le S^k   + \frac{\sqrt{\mu}}{1-\mu} \|\delta^{k}\|_1   \le c_1 \gamma^k \]
	with 
	\[ c_1=  (S^0 +\frac{\sqrt{\mu}}{1-\mu}\|\delta^0\|) \left( \frac{C_1+C_2}{C_1+C_2+1} \right)^{ -\frac{1}{2} }\quad \text{and}\quad \gamma =\sqrt{\frac{C_1+C_2}{C_1+C_2+1}},\] 
	which complets the proof of (ii).

	For part (iii), $\theta \in (\frac{1}{2},1)$.   Notice that 
	\[\frac{1-\theta}{\theta}< 1\ \text{ and }\  S^{k-2}-S^k+\|\delta^{k-1}\|_1 \to 0.\]
	Hence, there exists sufficient large $k$ such that 
	\[ S^{k-2}-S^k+\|\delta^{k-1}\|_1 \le \Big[S^{k-2}-S^k+\|\delta^{k-1}\|_1\Big]^{\frac{1-\theta}{\theta}},\]
	we assume the above inequality holds for $k \ge \bar k$.
	This,  combined with \eqref{eq.47rate}, yields 
	\[ S^k + \frac{\sqrt{\mu}}{1-\mu}\|\delta^k\|_1 \le (C_1+C_2) \Big[S^{k-2}-S^k+\|\delta^{k-1}\|_1\Big]^{\frac{1-\theta}{\theta}}.\]  This, combined with \eqref{eq:delta}, yields 
	\begin{equation}
	S^k+ \frac{\sqrt{\mu}}{1-\mu}\|\delta^k\|_1 \leq (C_1+C_2) \Big[S^{k-2}+\frac{\sqrt{\mu}}{1-\mu}\|\delta^{k-2}\|_1-(S^k+\frac{\sqrt{\mu}}{1-\mu}\|\delta^{k}\|_1)\Big]^{\frac{1-\theta}{\theta}}.
	\end{equation}
	Raising to a power of $\frac{\theta}{1-\theta}$ of both sides of the above inequality, we see
	\begin{equation}
	\Big[S^k+\frac{\sqrt{\mu}}{1-\mu}\|\delta^k\|_1\Big]^{\frac{\theta}{1-\theta}} \le C_3\Big[S^{k-2}+\frac{\sqrt{\mu}}{1-\mu}\|\delta^{k-2}\|_1-(S^k+\frac{\sqrt{\mu}}{1-\mu}\|\delta^{k}\|_1)\Big]
	\end{equation}
	with $C_3:=(C_1+C_2)^{\frac{\theta}{1-\theta}}$.

	Consider the ``even'' subsequence of $\{\bar k, \bar k+1, \ldots\}$ and define $\{\Delta_t\}_{t\ge N_1}$ with  $N_1:= \lceil \bar k/2 \rceil$, and 
	$\Delta_t:= S^{2t}+\frac{\sqrt{\mu}}{1-\mu}\|\delta^{2t}\|_1$. Then for all $t\ge N_1$, we have
	\begin{equation}
	\label{eq:delta1}
	\Delta_t^{\frac{\theta}{1-\theta}} \leq C_3(\Delta_{t-1}-\Delta_t)
	\end{equation}
	The remaining part of our proof is similar to    \cite[Theorem 2]{Attouch2009} (starting from  \cite[Equation (13)]{Attouch2009}).
	Define $h:(0,+\infty)\to \mathbb{R}$ by $h(s)=s^{-\frac{\theta}{1-\theta}}$ and let $T\in (1,+\infty)$. Take $k\geq N_1$ and consider   the case that $h(\Delta_k)\leq Th(\Delta_{k-1})$ holds. By rewriting 
	\eqref{eq:delta1}  as
	\[
	1\leq  C_3(\Delta_{k-1}-\Delta_{k})    \Delta^{-\frac{\theta}{1-\theta}}_{k},
	\]
	we obtain that 
	\[\begin{aligned}
	1 &\leq C_3(\Delta_{k-1}-\Delta_{k})h(\Delta_k)\\
	&\leq TC_3(\Delta_{k-1}-\Delta_{k})h(\Delta_{k-1})\\
	&\leq TC_3\int_{\Delta_k}^{\Delta_{k-1}}h(s)ds\\
	&\leq TC_3\frac{1-\theta}{1-2\theta}[\Delta_{k-1}^{\frac{1-2\theta}{1-\theta}}-\Delta_{k}^{\frac{1-2\theta}{1-\theta}}].
	\end{aligned}
	\]
	Thus if we set $u=\frac{2\theta-1}{(1-\theta)TC_3}>0$ and $\nu=\frac{1-2\theta}{1-\theta}<0$ one obtains that 
	\begin{equation}
	\label{eq:delta2}
	0<u \leq \Delta_{k}^{\nu}-\Delta_{k-1}^{\nu}.
	\end{equation}

	Assume now that $h(\Delta_k)>Th(\Delta_k)$ and set $q=(\frac{1}{T})^{\frac{1-\theta}{\theta}} \in (0,1)$. It follows immediately that $\Delta_k\leq q\Delta_{k-1}$ and furthermore - recalling that $\nu$ is negative - we have
	\[ 
	\Delta_k^{\nu} \geq q^{\nu}\Delta_{k-1}^{\nu} \quad\text{and}\quad 
	\Delta_k^{\nu}-\Delta_{k-1}^{\nu}  \geq (q^{\nu}-1)\Delta_{k-1}^{\nu}.
	\]
	Since $q^{\nu}-1>0$ and $\Delta_t\to 0^+$ as $t\to +\infty$, there exists $\bar{u}>0$ such that $(q^{\nu}-1)\Delta_{t-1}^{\nu}>\bar{u}$ for all $t\geq N_1$. Therefore we obtain that 
	\begin{equation}
	\label{eq:delta3}
	\Delta_{k}^{\nu}-\Delta_{k-1}^{\nu}\geq \bar{u}.
	\end{equation}
	If we set $\hat{u}=\min\{u, \bar{u}\}>0$, one can combine \eqref{eq:delta2}  and  \eqref{eq:delta3}  to obtain that 
	\[
	\Delta_k^{\nu}-\Delta_{k-1}^{\nu}\geq \hat{u}>0
	\]
	for all $k\geq N_1$. By summing those inequalities from $N_1$ to some $t$ greater than $N_1$ we obtain that $\Delta_t^{\nu}-\Delta_{N_1}^{\nu}\geq \hat{u}(t-N_1)$, implying 
	\begin{equation}
	\label{power.delta}
	\Delta_t\leq [\Delta_{N_1}^{\nu}+\hat{u}(t-N_1)]^{1/{\nu}}\leq C_4t^{-\frac{1-\theta}{2\theta-1}},
	\end{equation}
	for some $C_4>0$.
	
	As for the ``odd'' subsequence of $\{\bar k, \bar k+1, \ldots\}$, we can define $\{\Delta_t\}_{t\ge \lceil \bar k/2\rceil}$ with 
	$\Delta_t:= S^{2t+1}+\frac{\sqrt{\mu}}{1-\mu}\|\delta^{2t+1}\|_1$ and then can still show that \eqref{power.delta} holds true. 
	
	Therefore,  for all sufficiently large and even number  $k$,  
	\[ \|x^k-x^*\|_2 \leq  \Delta_{\frac{k}{2}}  \leq 2^{{\frac{1-\theta}{2\theta-1}} } C_4 k^{-\frac{1-\theta}{2\theta-1}}.\]  
	For all sufficiently large and odd number  $k$,  there exists $C_5 > 0$ such that 
	\[ \|x^k-x^*\|_2 \leq  \Delta_{\frac{k-1}{2}}  \leq 2^{{\frac{1-\theta}{2\theta-1}} } C_4 (k-1)^{-\frac{1-\theta}{2\theta-1}} \le   2^{{\frac{1-\theta}{2\theta-1}} } C_5 k^{-\frac{1-\theta}{2\theta-1}}.\]  
	Overall, we have 
	\[ \|x^k-x^*\|_2 \le    c_2 k^{-{\frac{1-\theta}{2\theta-1}} } \]
	where 
	\[c_2:=2^{{\frac{1-\theta}{2\theta-1}} } \max(C_4, C_5). \]
	This completes the proof of (iii).
	
\end{proof}
\section{Numerical results}

In this section, we demonstrate the local convergence rate of PIRL1 in practice. We test PIRL1 on 
 sparse signal recovery experiments  and observe its performance. The test problems can be formulated as 

\begin{equation}
\begin{aligned}
& \underset{x\in \mathbb{R}^n}{\min} 
& & \frac{1}{2}\|y-Ax\|^2_2 + \lambda \|x\|_p^p, 
\end{aligned}
\end{equation}
where  matrix $A\in \mathbb{R}^{m\times n}$ is generated uniformly at random from i.i.d. standard Gaussian   entries. We   set $L_f=1$ by orthonormalizing the rows of $A$.  The observation $y\in \mathbb{R}^m$ is generated by the sparse signal $x_{true}$ and Gaussian noise $e\sim \mathbb{N}(0, \sigma^2)$, that is $y=Ax_{true}+e$. The objective  of the experiments is to reconstruct a length $n$ sparse signal $x$ from $m$ observations. All experiments start from a   randomized initialized $x_0$ and use the same termination criterion  
\[
\max_{i = 1,...,n} | x_i\nabla_if(x)+\lambda p |x_i|^{p}| \leq \texttt{opttol}, 
\]
In the experiments, we set $\sigma = 10^{-3}, m=1024, n=2048$, and the $x_{true}$ contains $128$ randomly placed $\pm 1$ spikes.

Our purpose is to demonstrate the evolution of $\|x^k-x^*\|_2$ and verify whether the bounds \eqref{linear.1} and \eqref{linear.2} in Theorem
\ref{th:convergerate} can be witnessed. Therefore, we first run the algorithm with sufficiently small tolerance
 $\texttt{opttol} = 10^{-12}$ and  the final iterate $x^*$ used as the surrogate  of the real solution $x_{opt}$, since in this case
  $\|x_{opt} -x^*\|_2$ is deemed sufficiently small.  
 In this way, we can use 
 \[ \|x^k-x^*\|_2 \le \|x^k-x_{opt}\|_2 + \|x_{opt} -x^*\|_2 \approx \|x^k-x_{opt}\|_2\]
 to examine the local behavior of PIRL1. 
The algorithm is rerun with $ \texttt{opttol} = 10^{-8}$ meaning $x^k$ is sufficiently close to the optimal solution. 
     The last 100 iterations of  $\log_{10}(\|x^k-x^*\|_2)$ is plotted  in Figure \ref{fig.local}, where $T$ represents the last 
     iteration for each run. 
 To see how performance is affected by different $p$ and $\lambda$,  we repeat this procedure   for cases with $p=0.2, 0.5, 0.8$ and  $\lambda=0.0001, 0.001, 0.01$.  In each case, we randomly generate 5 problems.

\begin{figure}[htbp] 
	\centering 
	\subfigure[$\lambda =0.0001$, $p=0.2$.]{\includegraphics[width=1.5in]{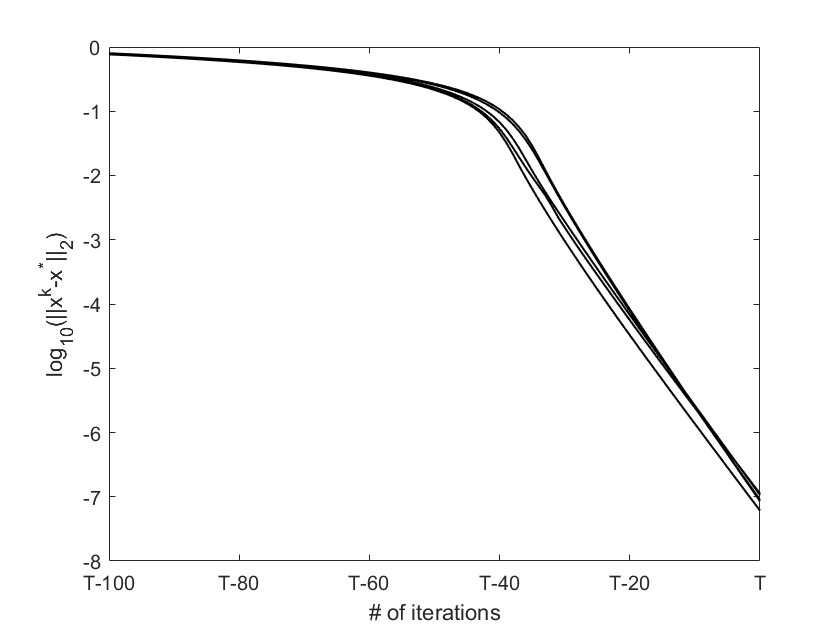}}
	\subfigure[$\lambda =0.0001$, $p=0.5$.]{\includegraphics[width=1.5in]{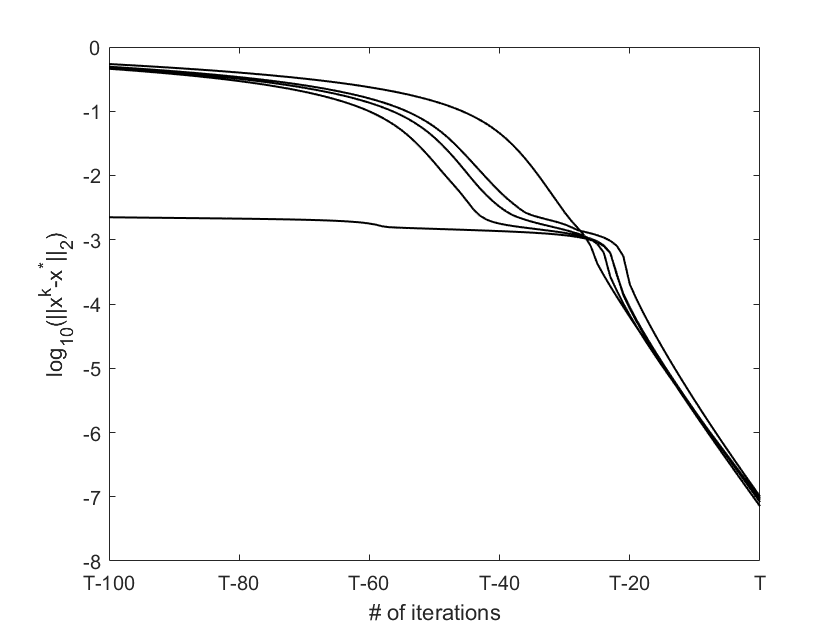}}
	\subfigure[$\lambda =0.0001$, $p=0.8$.]{\includegraphics[width=1.5in]{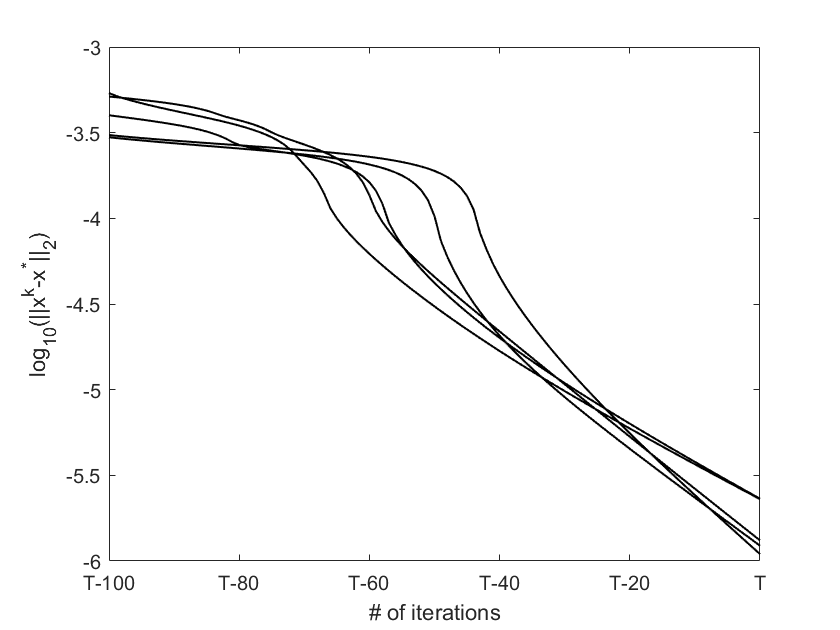}}
	\subfigure[$\lambda =0.001$, $p=0.2$.]{\includegraphics[width=1.5in]{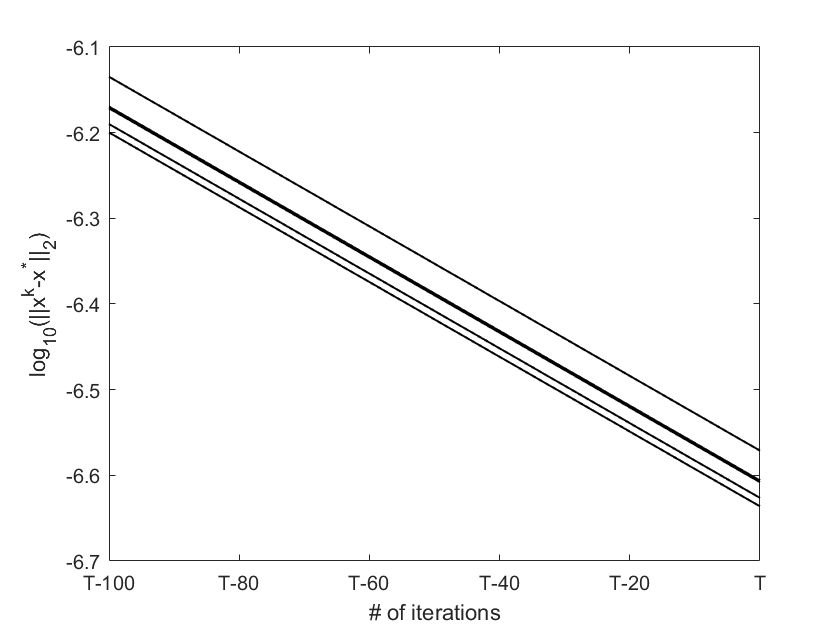}}
	\subfigure[$\lambda =0.001$, $p=0.5$.]{\includegraphics[width=1.5in]{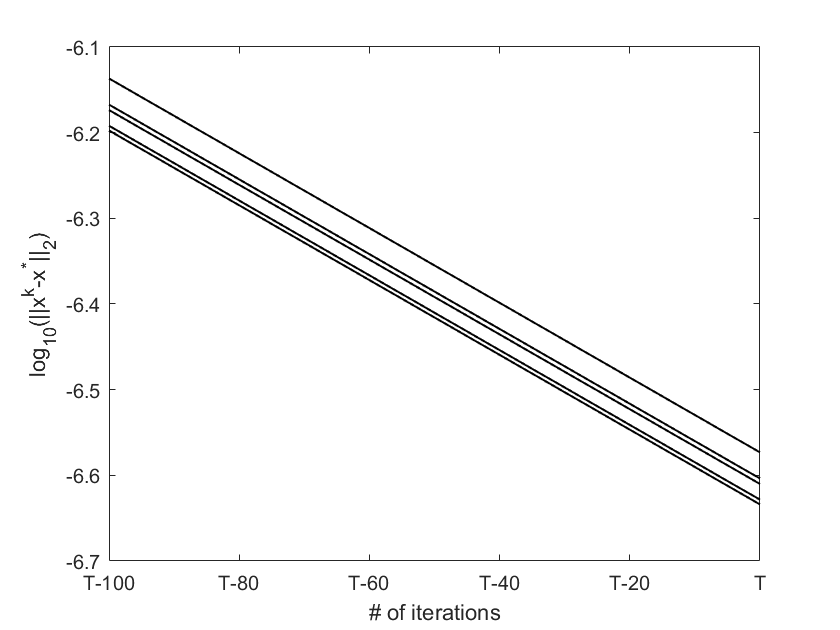}}
	\subfigure[$\lambda =0.001$, $p=0.8$.]{\includegraphics[width=1.5in]{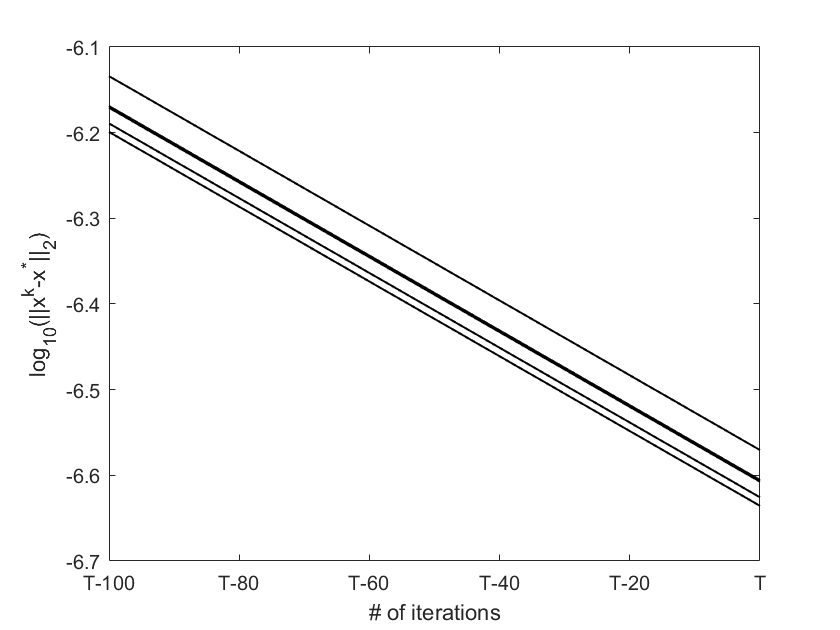}}
	\subfigure[$\lambda =0.01$, $p=0.2$.]{\includegraphics[width=1.5in]{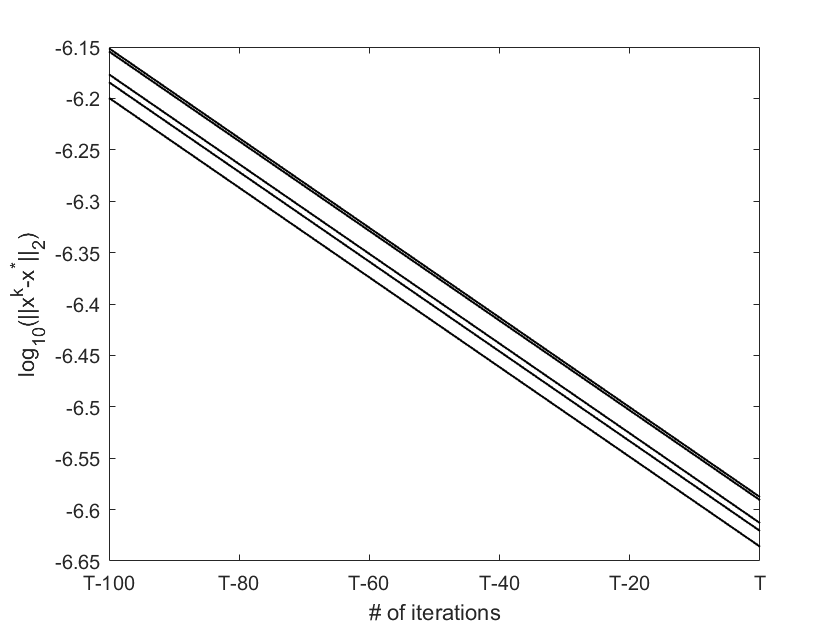}}
	\subfigure[$\lambda =0.01$, $p=0.5$.]{\includegraphics[width=1.5in]{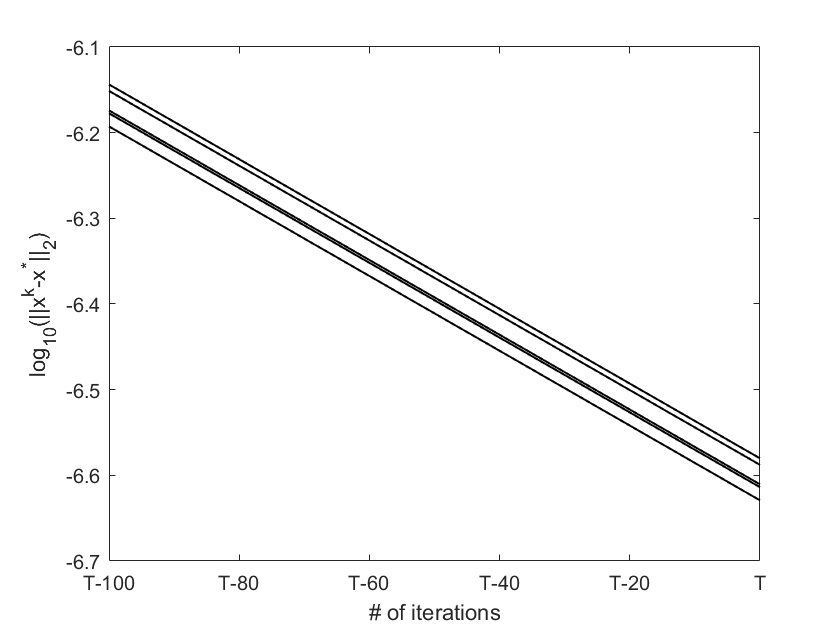}}
	\subfigure[$\lambda =0.01$, $p=0.8$.]{\includegraphics[width=1.5in]{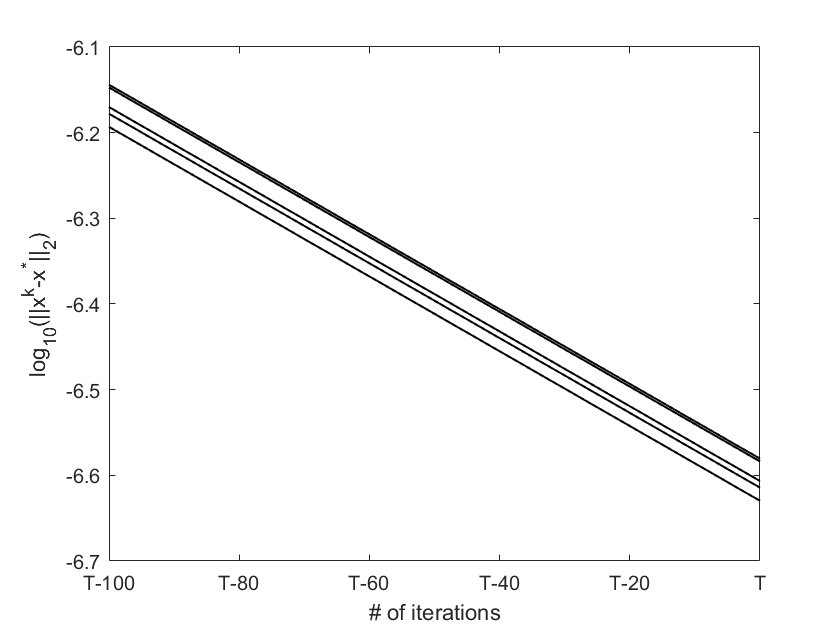}}
	\caption{The last 100 iterations of $\log_{10}(\|x^k-x^*\|_2)$.} 
	\label{fig.local}
\end{figure}

From Figure \ref{fig.local}, we can see that PIRL1 exhibits linear convergence for all problems in all cases, meaning the bound \eqref{linear.1} 
is always witnessed.  
This may indicate that  the KL property can be satisfied for a wide range of test problems, and 
PIRL1 can then achieve local linear convergence in many cases. 

\section{Conclusion} 

In this paper, we have analyzed the global convergence and local convergence rate of the proximal iteratively reweighted $\ell_1$ method for solving $\ell_p$ 
regularization problems under the KL property.  We have shown that the iterates generated by this method have a unique limit point. It  has a locally linear convergence or sublinear convergence under KL property. It should be noticed that our analysis can be easily extended to 
other types of nonconvex  regularization problems under the assumption of the KL property for the loss function.

 \section*{Acknowledgements}

  Hao Wang was supported  by the Young Scientists Fund of the National Natural Science Foundation of China under Grant 12001367.

\bibliographystyle{spmpsci}      
\bibliography{references.bib}   


\end{document}